\documentclass[12pt,a4paper]{article}
\usepackage[utf8]{inputenc}
\usepackage{amsmath,amsthm}
\usepackage[dvipsnames]{xcolor}
\usepackage{amsfonts}
\usepackage{amssymb}
\usepackage{tikz}
\usepackage{tikz-cd}
\usepackage{caption}
\usepackage[nottoc]{tocbibind}
\usepackage[toc,page]{appendix}
\usepackage{listings}
\usepackage{pdfpages}
\usepackage{mathtools}
\usepackage[all,cmtip]{xy}
\usepackage{lipsum}
\usepackage{breqn}
\usepackage{subcaption}
\usepackage{hyperref}
\usepackage{tkz-euclide}
\usepackage{pgfplots}
\usepackage[sort,nocompress]{cite}

\usetikzlibrary{angles,quotes}
\usetikzlibrary{decorations.markings}
\usetikzlibrary{hobby}

\usepackage{authblk}

\newtheorem{claim}{Claim}
\newtheorem{theorem}{Theorem}
\newtheorem{definition}{Definition}
\newtheorem{lemma}{Lemma}

\theoremstyle{remark}
\newtheorem*{remark}{Remark}
\newtheorem*{remarks}{Remarks}

\numberwithin{lemma}{section}
\numberwithin{claim}{section}
\numberwithin{definition}{section}
\numberwithin{proposition}{section}
\numberwithin{equation}{section}

\def\N{\mathbb{N}}

\def\R{\mathbb{R}}

\def\<={\leq}
\def\>={\geq}

\def\bd{\partial}

\def\l{\ell}

\def\diam{\mathrm{diam}}
\def\im{\mathrm{Im}\,}

\newcommand{\set}[1]{\left\lbrace#1\right\rbrace}
\newcommand\inprod[2]{\langle{#1},{#2}\rangle}

\title{An analogue of the Blaschke--Santal\'{o} inequality for billiard dynamics}

\author[1]{Daniel Tsodikovich\thanks{   School of Mathematical Sciences,
   Tel Aviv University, Tel Aviv 69978, Israel.
    \textit{E-mail}: \href{mailto:tsodikovich@tauex.tau.ac.il}{\texttt{tsodikovich@tauex.tau.ac.il}}}}

\date{}
\begin{document}
\maketitle
\begin{abstract}
The Blaschke--Santal\'{o} inequality is a classical inequality in convex geometry concerning the volume of a convex body and that of its dual. 
In this work we investigate an analogue of this inequality in the context of billiard dynamical system:
we replace the volume with the length of the shortest closed billiard trajectory.
We define a quantity called the ``billiard product" of a convex body $K$, which is analogous to the volume product studied in the Blaschke--Santal\'{o} inequality.
In the planar case, we derive an explicit expression for the billiard product in terms of the diameter of the body.
We also investigate upper bounds for this quantity in the class of polygons with a fixed number of vertices.
\end{abstract}

\section{Introduction and main results}\label{section:intro}
The lengths of closed billiard trajectories inside a convex body (convex compact set with non-empty interior) are well studied quantities. 
For example, it is known that these lengths are related to the Dirichlet eigenvalues of the body (see, e.g., \cite{AnderssonK.G1977Tpos, HuangGuan2018Otml}).
The length of a shortest closed billiard trajectory inside a convex domain is also known to be related to symplectic capacities, as was shown, e.g., in \cite{ViterboClaude2000MaIP,Artstein-AvidanShiri2014BfMB}.
This was used in \cite{Artstein_Avidan_2014} to show that a volume-capacity type conjecture by Viterbo implies the well-known Mahler conjecture regarding the volume product of convex bodies.
For convex domains with smooth boundary, the existence of closed billiard orbits with any number of bounce points is well known (see, e.g., \cite{tabachnikov2005geometry}).
For other convex domains, the answer is not so clear.
For example, the existence of any closed billiard orbit in an obtuse triangle is still an open question. 
 For this reason, in what follows, we consider \textit{generalized} billiard trajectories, in which the particle is allowed to pass also through corners of the boundary.
At a corner, the orbit is required to make equal angles with one of the supporting hyperplanes to the boundary (for the precise definition, see Definition \ref{def:genBilliardTrajectory} at Subsection \ref{subsection:shortestBT}).
This way, every obtuse triangle has a closed generalized billiard trajectory, namely, the altitude from the vertex of the obtuse angle to the opposite edge.
We denote by $\alpha(K)$ the length of a shortest closed generalized billiard trajectory inside the convex (not necessarily bounded) set $K$.
 If $K$ does not have a closed billiard trajectory,  we set $\alpha(K)=\infty$.
 In \cite{BezdekDaniel2009Sbt, GhomiMohammad2004Spbt} it was shown that in every convex body $K$ there always exists a closed billiard trajectory, and moreover, there always exists a shortest one.
 In light of this, the quantity $\alpha(K)$ is well defined.
 
We recall the classical construction of the polar body from convex geometry: for a convex body $K\subseteq \R^n$ and $z\in \R^n$, the dual body to $K$ with respect to $z$ is
\[K^z=\set{y\in\R^n\mid\forall x\in K: \inprod{x-z}{y}\leq 1},\]
where $\langle\cdot,\cdot\rangle$ denotes the Euclidean standard inner product.
The point $z$ is referred to as the center of duality.
In this work, we analyze the \textit{billiard product} of a convex body $K\subseteq\R^n$:
\[\beta(K):=\inf\limits_{z\in\R^n} \alpha(K)\alpha(K^z).\]
This product is an analogue of the classical volume product: the infimum of the product of the volume of a convex body and the volume of its dual, where the infimum is taken over all possible centers of duality. 
The volume product is studied extensively in convex geometry.
An upper bound for the volume product was obtained, and is known as the Blaschke--Santal\'{o} inequality  (see, e.g., \cite{Santalo1949,MR670798,blaschke1917affine}).
On the other hand, the lower bound of the volume product is still open in general, and is known as the Mahler conjecture \cite{mahler1939minimalproblem}. 
Other properties of the volume product have been investigated in the past, see for example \cite{KaiserM.J1993Tspo, AlexanderMatthew2019PoMV, MeyerMathieu2011Otvp}.
We remark that another analogue of the volume product was studied in \cite{BucurDorin2016BaMi}.
In that work, the volume product was replaced with the $\lambda_1$ product: the product of the smallest Dirichlet eigenvalue in $K$ and in its dual.

It is not hard to check that the billiard product $\beta(K)$ is invariant under similarities, that is, compositions of isometries and homotheties.
In this work, while some of the results apply to arbitrary dimension, we will focus on the two-dimensional case, where billiard orbits are best understood.
First, we identify the analogue of a ``Santal\'{o} point" in this setting --- a point which minimizes $\alpha(K^z)$.
It is known (see, e.g., \cite{Santalo1949, KaiserM.J1993Tspo}) that a point that minimizes the volume of $K^z$ is a point $z$ for which $K^z$ has the origin as the center of mass.
In our case we get a different result, which we now state.
Recall that the diameter of a convex body $K$, denoted by $\diam(K)$, is the maximal length of a line segment that is contained in $K$.
\begin{theorem}\label{thm:santaloPtIdentity}
Let $K$ be a two-dimensional convex body. 
Then \[\inf\limits_{z\in \R^2} \alpha(K^z)=\frac{8}{\diam(K)},\]
and the infimum is attained when $z$ is the midpoint of a diameter of $K$.
Moreover, for all dimensions, $8 \slash \diam (K)$ is an upper bound for $\inf\limits_{z\in\R^n} \alpha(K^z)$.
\end{theorem}
\begin{remark}
Theorem \ref{thm:santaloPtIdentity} implies, in particular, that for two-dimensional convex bodies $K$,
\[\beta(K)=\frac{8\alpha(K)}{\diam(K)}.\]
\end{remark}

As a result, we derive an analogue of the Blaschke--Santal\'{o} inequality.
Recall that the width of a convex body $K\subseteq \R^n$ in the direction of the vector $v\neq 0$ is the distance between the two supporting hyperplanes to $K$ that are orthogonal to $v$.
The convex body $K$ is said to have \textit{constant width} if it has the same widths in all directions.
\begin{theorem}\label{thm:santaloAnalogue}
If $K\subseteq\R^n$ is a convex body and $B$ is the unit ball, then \[\beta(K)\leq \beta(B)=16.\]
Moreover, if $\beta(K)=16$, then $K$ must have constant width.
\end{theorem}
\begin{remarks}
\begin{enumerate}
\item In the planar case we show moreover that for any convex body $K$ of constant width $\beta(K)=16$.
This is in contrast to the volume product, for which the only maximizers are ellipsoids.
\item It is conjectured that also in arbitrary dimension, the shortest billiard trajectories in bodies of constant width have two bounce points.
If this conjecture is true, and if $\frac{8}{\diam(K)}$ is still a lower bound for $\alpha(K^z)$ in arbitrary dimension, then the previous remark applies to arbitrary dimension, and gives the following characterization of bodies with constant width in $\R^n$: these are exactly the bodies $K$ for which $\beta(K)=16$.
\item It is natural to also consider lower bounds, and look for Mahler-like inequalities. 
In this case, it is simple to see that the billiard product can be arbitrarily small: if $K$ is a thin and long rectangle, then $\beta(K)$ can get arbitrarily close to zero.
\end{enumerate}
\end{remarks}
We turn to investigate the maximum of the billiard product over the set of polygons with a fixed number of vertices. 
For the volume (area) product, it is known (see e.g. \cite{MeyerMathieu2011Otvp, AlexanderMatthew2019PoMV}) that the maximal volume product is attained for regular polygons.
We find the maximizers in the class of triangles:
\begin{theorem}\label{thm:optimalTriangle}
For a triangle $T$, one has $\beta(T)\leq 12$.
Moreover, equality is attained if and only if $T$ is a regular triangle.
\end{theorem}
We remark that in general the regular polygons are not those that maximize the billiard product.
In Subsection \ref{subsec:billiardPolygon} we show this for regular polygons with even number of vertices.
We also discuss there a potential maximizer for $\beta$ in the class of quadrilaterals.

It is known (see, e.g., \cite[Proposition 1.1.15]{Artstein-AvidanShiri2015AGAP}) that the Steiner symmetrization can be used to prove the Blaschke--Santal\'{o} inequality.
In light of this, it is natural to consider a similar approach here.
We show that while in general Steiner symmetrization may decrease the billiard product, the Steiner symmetrization of a triangle in the direction of an altitude always increases the billiard product.
As a consequence of that, we will obtain an alternative proof of Theorem \ref{thm:optimalTriangle}, using the Steiner symmetrization.
\medskip

\textbf{Structure of the paper:} In Section \ref{section:bg} we recall some known facts about the length of a shortest closed billiard trajectory inside a convex body, and about the dual of a convex body. 
In Section \ref{section:santalopt} we prove Theorem \ref{thm:santaloPtIdentity}, and find the analogue of a Santal\'{o} point in our setting.
In Section \ref{section:polygonOptimization} we discuss upper bounds for the billiard product.
In Subsection \ref{subsec:globalUpperBound} we prove Theorem \ref{thm:santaloAnalogue}.
In Subsection \ref{subsec:billiardPolygon} we prove Theorem \ref{thm:optimalTriangle}, and give some heuristics about an optimal quadrilateral.
Finally, in Subsection \ref{subsec:steinerSymm} we discuss the effect of the Steiner symmetrization on the billiard product, and give an alternative proof of Theorem \ref{thm:optimalTriangle}.
\section*{Acknowledgments}
This paper is part of the author's Ph.D. thesis, being carried out under the joint supervision of Professor Misha Bialy and Professor Yaron Ostrover at Tel Aviv University.
I would like to thank Professor Roman Karasev for his remarks.
I would also like to thank Itai Bar-Deroma, Arnon Chor, Daniel Hadas, and Leonid Vishnevsky for their comments and discussions.
The author is supported by ISF grants 580/20, 667/18, and by DFG grant MA-2565/7-1 within the Middle East Collaboration Program.
\section{Preliminaries}\label{section:bg}
In this section we review some standard well-known facts about the length of a shortest closed billiard trajectory inside a  convex body, and about the construction of the polar dual to a convex body.
\subsection{Shortest billiard trajectory}\label{subsection:shortestBT}
In this subsection we will present some known facts about the length of a shortest closed billiard trajectory inside a convex body $K\subseteq\R^n$. 
We begin with the definition of this notion.
Recall that a \textit{supporting hyperplane} to a convex closed set $K$ is a hyperplane $H$ which has non-empty intersection with $K$, and for which $K$ is contained in either of the closed half-spaces determined by $H$.
We start from the following classical definition (cf. \cite{GhomiMohammad2004Spbt, BezdekDaniel2009Sbt}).
 \begin{definition}\label{def:genBilliardTrajectory}
 Let $K\subseteq\R^n$ be a convex closed set.
 A piecewise linear curve $\gamma$ contained in $K$ is called a \textup{generalized billiard trajectory of $K$}, if the points where $\gamma$ is not smooth are all in $\bd K$, and at each such point $\gamma(t)$ there exists a supporting hyperplane $H$ to $K$ for which:
 \begin{enumerate}
 \item The normal to $H$ and the one-sided derivatives of $\gamma$ at the point $\gamma(t)$ are coplanar.
 \item The one-sided derivatives of $\gamma$ at the point $\gamma(t)$ make equal angles with $H$.
 \end{enumerate}
 \end{definition}
For a convex body $K\subseteq \R^n$, we denote by $\alpha(K)$ the length of its shortest closed billiard trajectory.
In \cite{GhomiMohammad2004Spbt} Ghomi provided a lower bound for $\alpha(K)$ in terms of the inradius of $K$, that is, the largest radius of a ball contained in $K$.
In \cite{BezdekDaniel2009Sbt}, Bezdek and Bezdek investigated the shortest length of a closed generalized billiard trajectory from a different perspective. 
They gave the following characterization of the shortest length of a billiard trajectory.
 For $m\in\N$, let $P_m(K)$ denote the set of all polygonal paths with $m$ vertices that cannot be translated into the interior of $K$.
For a polygonal path $P$ with vertices $q_1,...,q_m$ , let $\l(P)$ denote its perimeter:
\[\l(P)=\sum_{i=1}^m |q_{i+1}-q_i|,\]
where the indices are considered cyclically, $q_{m+1}=q_1$, and $|\cdot|$ denotes the Euclidean norm. 
It was proved in \cite{BezdekDaniel2009Sbt} that
\begin{equation}\label{eq:bezdekbezdekresult}
\alpha(K)=\min\limits_{m\leq n+1}\min\limits_{P\in P_m(K)}\l(P),
\end{equation}
and furthermore, any polygonal path that minimizes the right-hand side is a translate of a shortest billiard orbit of $K$. 
This implies, in particular, the following \textit{monotonicity property} for $\alpha$: if $K\subseteq T$ are two convex bodies, then $\alpha(K)\leq \alpha(T)$ (since if a polygonal path cannot be translated into the interior of $T$, then it cannot be translated into the interior of $K$, so $\alpha(T)$ is obtained by minimizing over a smaller set). The monotonicity property was also shown in \cite{Artstein-AvidanShiri2014BfMB} with a different method.
This result also implies that a shortest closed billiard trajectory for a convex body $K\subseteq \R^n$ can be chosen to have at most $n+1$ bounce points.
We will call a closed billiard orbit with $m$ bounce points an $m$-orbit.
We recall how the shortest billiard trajectory in a triangle is determined.
This was addressed in \cite[Proposition 4.1]{AlkoumiNaeem2015Scbo}.
For completeness, we restate this result, and we also include a convenient formula for the Fagnano orbit (the orbit connecting the three feet of the altitudes of an acute triangle) which is used in sections \ref{section:santalopt} and \ref{section:polygonOptimization}.
\begin{lemma}\label{lem:alphaTriangle}
Let $T\subseteq\R^2$ be a triangle. 
If $T$ is obtuse (or has a right angle), then $\alpha(T)$ is a 2-orbit realized by the altitude from the obtuse (or right) angle.
If $T$ is acute, then $\alpha(T)$ is a 3-orbit (the so-called Fagnano orbit) that connects the feet of the altitudes of $T$. 
In this case, $\alpha(T)=2h \sin\theta$, where $h$ is the length of any of the altitudes of\, $T$, and $\theta$ is the angle at the vertex from which that altitude is dropped.
In both cases, these are the only orbits that realize $\alpha(T)$.
\end{lemma}
\begin{proof}
The fact that the shortest billiard trajectory in $T$ is a 3-orbit if and only if $T$ is acute follows from well-known classical  results (see, for example, \cite[Proposition 9.4.1.3]{BergerMarcel1987G} or \cite[Lemmas 2.1,2.2]{AlkoumiNaeem2015Scbo}).
If $T$ is an obtuse (or right) triangle, then the shortest billiard trajectory cannot be a 3-orbit, and by the results of \cite{BezdekDaniel2009Sbt}, it then must be a 2-orbit. 
On the other hand, a 2-orbit must be an altitude between a vertex and the opposite side, and the only altitude that is contained inside $T$ is the one dropped from the obtuse (or right) angle.

If $T$ is acute, then the unique  billiard 3-orbit in $T$ is the Fagnano orbit.
We use the following classical formula for the length of the Fagnano orbit \cite[p. 191]{johnson1929modern}: $\alpha(T)=2\,\mathrm{area}(T)\slash R$, where $R$ is the circumradius of $T$.
By the sine theorem, $2R=a\slash\sin\theta$, where $a$ is a side of $T$ and $\theta$ is the angle opposite to this side.
Also, $\mathrm{area}(T)=ha\slash 2$, where $h$ is the altitude dropped from the angle $\theta$ to the side $a$.
As a result, \[\alpha(T)=ha\Big\slash\frac{a}{2\sin\theta}=2h\sin\theta,\] which is the required result.
\end{proof}

In what follows, we also require some understanding of billiard orbits in unbounded sets.
Some unbounded convex sets, like strips, can have closed billiard orbits, but not all unbounded convex sets have closed billiard orbits (cones are counter-examples).
While the results of \cite{BezdekDaniel2009Sbt} were formulated only for convex bodies, one can also use them to deduce something about unbounded convex sets.
\begin{lemma}\label{lem:unboundedImpossible}
Let $K\subseteq\R^n$ be a closed convex set with non-empty interior (maybe unbounded). 
For $R>0$, let $K_R$ be the intersection of $K$ with the ball of radius $R$ around the origin.
If $\set{\alpha(K_R)\mid R>0}$ is unbounded, then $K$ does not contain a closed billiard orbit and hence, $\alpha(K)=\infty$.
\end{lemma}
\begin{proof}
The proof is by contradiction. 
Suppose that $q_1,...,q_m$ is a closed generalized billiard orbit of $K$.
Denote its perimeter by $L$.
Then there exists $R_0>0$ such that for $R>R_0$ all of the points $q_1,...,q_m$ are in $K_R$.
Since these points form a generalized billiard trajectory, the inner angle bisector at each $q_i$ is orthogonal to one of the supporting hyperplanes of $K$ at $q_i$.
But a supporting hyperplane to $K$ at $q_i$ is also a supporting hyperplane to $K_R$ at $q_i$, so these points form a generalized billiard orbit in $K_R$, and hence $\alpha(K_R)\leq L$. 
Since $\alpha(K_R)$ is monotone with respect to $R$, it follows that $\alpha(K_R)\leq L$ holds also for $R\leq R_0$, and we get a contradiction to the assumption that $\set{\alpha(K_R)\mid R>0}$ is unbounded.
Consequently, $K$ does not contain a closed generalized billiard orbit.
\end{proof}

In what follows, we also use the fact that $\alpha$ is continuous with respect to the Hausdorff topology.
This fact can be derived from \cite[Theorem 2.13]{Artstein-AvidanShiri2014BfMB}, which relates $\alpha(K)$ to the Ekeland-Hofer-Zehnder capacity of the product of $K$ and a ball.
For completeness, we give an alternative proof, that relies on the characterization of $\alpha$ from \cite{BezdekDaniel2009Sbt}.
\begin{lemma}\label{lem:alphaCont}
The function $K\mapsto\alpha(K)$ is continuous with respect to the Hausdorff topology on the space of convex bodies in $\R^n$.
\end{lemma}
\begin{proof}
We show, in fact, that $\alpha$ is a $2(n+1)$-Lipschitz function.
Let $K,T\subseteq\R^n$ be two convex bodies, and let $r$ be their Hausdorff distance.
This means that $K\subseteq T+rB$ and $T\subseteq K+rB$, where $B$ is the unit ball of $\R^n$.
By symmetry, it is enough to show the inequality $\alpha(K)\leq \alpha(T)+2r(n+1)$.
Since $\alpha$ is monotone with respect to inclusion, $\alpha(K)\leq\alpha(T+rB)$, so it is enough to show the inequality 
\begin{equation}\label{eq:LipchitzAlpha}
\alpha(T+rB)\leq \alpha(T)+2r(n+1).
\end{equation}
Consider a shortest billiard trajectory in $T$, and denote its bounce points by $q_1,...,q_m$, so that $\alpha(T)=\sum_{i=1}^m |q_{i+1}-q_i|$.
From \eqref{eq:bezdekbezdekresult} we know that $m\leq n+1$.
Write $v_i=\frac{q_{i+1}-q_i}{|q_{i+1}-q_i|}$ for the unit direction vectors of the billiard orbit.
The fact that $q_1,...,q_m$ is a billiard orbit, means that $v_i,v_{i+1}$ make equal angles with a supporting hyperplane to $T$ at $q_{i+1}$. 
Let $n_{i+1}$ be the outer unit normal to this hyperplane.
We will now define a configuration of points that cannot be translated into the interior of $T+rB$; this will show the inequality \eqref{eq:LipchitzAlpha}.
Consider the configuration $\tilde{q}_i=q_i+rn_i\in T+rB$.
First, estimate the perimeter of this configuration:
\begin{multline}\label{eq:lengthEstimate}
\sum_{i=1}^m|\tilde{q}_{i+1}-\tilde{q}_i|=\sum_{i=1}^m|q_{i+1}+rn_{i+1}-q_i-rn_i|\leq\sum_{i=1}^m \Big(|q_{i+1}-q_i|+r|n_{i+1}-n_i|\Big)\leq\\
 \leq \sum_{i=1}^m \Big(|q_{i+1}-q_i| + 2r\Big)=\alpha(T)+2rm\leq\alpha(T)+2r(n+1).
\end{multline}
Thus, if we show that the configuration $\tilde{q}_1,...,\tilde{q}_m$ cannot be translated into $T+rB$, then \eqref{eq:bezdekbezdekresult} will yield \eqref{eq:LipchitzAlpha}.
First observe that $n_i$ is a normal to a supporting hyperplane of $T$ at $q_i$ as well as to a supporting hyperplane of $rB$ at $rn_i$, so it is a normal to a supporting hyperplane of $T+rB$ at $q_i+rn_i$. 
Using that, we resort to an argument that was employed, for example, in \cite[Theorem 2.1]{AkopyanArseniy2016Eatc}.
The billiard law implies that there exist positive scalars $\lambda_i$ such that $v_{i+1}-v_i=-\lambda_{i+1} n_{i+1}$, for $i=1,...,m$.
Summing over $i=1,...,m$, we get that 
\begin{equation}\label{eq:zeroConvexCombination}
\sum\limits_{i=1}^m\lambda_i n_i = 0.
\end{equation}
Suppose by contradiction that there exists a translation vector $t\in\R^n$ such that all the points $\tilde{q}_i+t$ lie in the interior of $T+rB$.
From the definition of a supporting hyperplane, it follows that the function $x\mapsto \inprod{n_i}{x}$ attains its maximum in $T+rB$ at the point $\tilde{q}_i$.
This maximum must be attained strictly on the boundary of $T+rB$, so
\[\inprod{n_i}{\tilde{q}_i+t}<\inprod{n_i}{\tilde{q}_i}\]
for all $i=1,...,m$, from which it follows that $\inprod{n_i}{t}<0$.
Now, we use the fact that the $\lambda_i$ are positive, and so $\sum\limits_{i=1}^m\lambda_i\inprod{n_i}{t}<0$, which contradicts \eqref{eq:zeroConvexCombination}.
Therefore, $\tilde{q}_1,...,\tilde{q}_m$ cannot be translated into the interior of $T+rB$, which together with the estimate \eqref{eq:lengthEstimate} implies \eqref{eq:LipchitzAlpha}, and completes the proof.

\end{proof}
\begin{remark}
A consequence of Theorem \ref{thm:santaloPtIdentity} and Lemma \ref{lem:alphaCont} is that the function $K\subseteq\R^2\mapsto\beta(K)$ is continuous with respect to the Hausdorff topology, since both $\alpha$ and $\diam$ are.
\end{remark}

\subsection{The polar (dual) body}\label{subsection:dual}
Recall that given a convex body $K\subseteq\R^n$ and $z\in \R^n$, the polar body of $K$ with respect to the center of duality $z$ is defined by
\[K^z=(K-z)^0=\set{y\in\R^n\mid\forall x\in K: \inprod{x-z}{y}\leq 1}.\]
\begin{figure}
\centering
\begin{tikzpicture}[scale = 1]
\draw (2,0)--(2,2)--(0,2)--(0,0)--(2,0);
\tkzDefPoint(0.5,0.5){A};
\tkzDrawPoint(A);
\draw[domain = 0:360,smooth, variable = \t, dashed,blue] plot({0.5+cos(\t)},{0.5+sin(\t)});
\tkzDefPoint(-0.5,-0.5){AA};
\draw[red](-0.5+1,-0.5-1)--(-0.5-1,-0.5+1);
\tkzDrawSegment[dashed](A,AA);
\tkzDefPoint(0.833,0.833){BB};
\draw[red](0.833+0.38,0.833-0.38)--(0.833-0.38,0.833+0.38);
\tkzDrawSegment[dashed](A,BB);
\tkzDefPoint(0.3,1.1){CC};
\draw[red](0.3+3*0.08,1.1+1*0.08)--(0.3-3*0.6,1.1-1*0.6);
\tkzDrawSegment[dashed](A,CC);
\tkzDefPoint(1.1,0.3){DD};
\draw[red](1.1+1*0.08,0.3+3*0.08)--(1.1-1*0.6,0.3-3*0.6);
\tkzDrawSegment[dashed](A,DD);
\node[above] at (2,2) {$K$};
\node [below,blue] at (0.5,0.5) {$z$};
\node [left,red] at (-0.5,-0.5){$K^z+z$};
\end{tikzpicture}
\caption{The dual of the polygon $K$ with respect to the point $z$. 
The blue dashed curve is the unit circle centered at $z$.
 The dashed lines emanating from $z$ point to the vertices of $K$.
  The red polygon is $K^z+z$, and its edges are perpendicular to those dashed lines. \label{fig:dualPolygon}}
\end{figure}
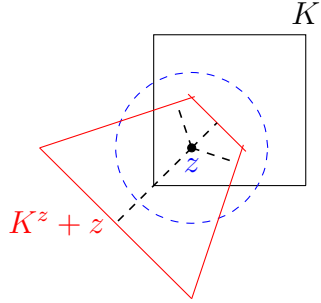

This is a classical construction in convex geometry.
The set $K^z$ is always closed and convex.
It is well known that if $z$ is an interior point of $K$, then $K^z$ is also a convex body.
However, if $z\in\bd K$ or $z\not\in K$, then $K^z$ is unbounded. 
We will now recall how to construct the set $K^z$ where $K$ is a polygon and $z$ is an interior point of $K$.
For that, we first recall what the inversion of a point about a circle is:
suppose that $O$ is the center of a circle of radius $R$, and $A$ is a point in the plane. 
Then the reflection of $A$ about this circle is the point $B$, which is on the ray from $O$ to $A$ and for which $|OA|\cdot |OB| = R^2$.
Now, if $K\subseteq \R^2$ is a convex polygon and $z$ is an interior point of $K$, then $K^z$ is constructed in the following way.
Invert all the vertices of $K$ with respect to the unit circle around $z$.
Then take the lines that pass through these points and are perpendicular to the lines connecting $z$ and the vertices: these will be the edges of $K^z$.
The result will be again a polygon with $n$ vertices, see Figure \ref{fig:dualPolygon}.

For triangles $T$ we also make a distinction between points $z$ in the interior of $T$ for which the dual $T^z$ is obtuse, and those for which it is acute.
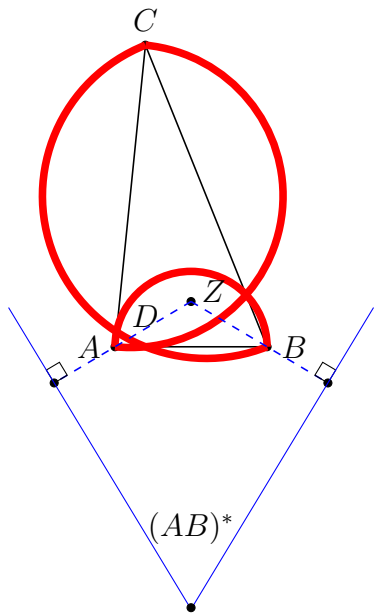
\begin{figure}
\centering
\begin{tikzpicture}[scale = 2]
\tkzDefPoint(0,0){A};
\tkzDefPoint(1,0){B};
\tkzDefPoint(0.2,2){C};
\tkzDrawPoints(A,B,C);
\tkzDrawSegment(A,B);
\tkzDrawSegment(B,C);
\tkzDrawSegment(C,A);
\tkzDrawSemiCircle[diameter,line width = 1mm, red](A,B);
\tkzDrawSemiCircle[diameter, line width = 1mm, red](B,C);
\tkzDrawSemiCircle[diameter, line width = 1mm, red](C,A);
\node at (0.2,0.2) {$D$};

\node[left] at (0,0) {$A$};
\node[right] at (1,0) {$B$};
\node[above] at (0.2,2) {$C$};

\tkzDefPoint(0.5,0.3){Z};
\tkzDrawPoint(Z);
\node[above] at (0.65,0.2) {$Z$};
\tkzDefPoint(0.5-1.8*0.5,0.3-1.8*0.3){AA};
\tkzDefPoint(0.5+1.8*0.5,0.3-1.8*0.3){BB};
\tkzDrawPoints(AA,BB);
\tkzDrawSegment[dashed,blue](Z,AA);
\tkzDrawSegment[dashed,blue](Z,BB);
\tkzDefPoint(0.5-1.8*0.5-0.1*3,0.3-1.8*0.3+0.1*5){gA};
\tkzMarkRightAngle[size = 0.1](Z,AA,gA);
\tkzDefPoint(0.5+1.8*0.5+0.1*3,0.3-1.8*0.3+0.1*5){gB};
\tkzMarkRightAngle[size=0.1](Z,BB,gB);
\draw[blue](0.5-1.8*0.5-0.1*3,0.3-1.8*0.3+0.1*5)--(0.5-1.8*0.5+0.3*3,0.3-1.8*0.3-0.3*5);
\draw[blue](0.5+1.8*0.5+0.1*3,0.3-1.8*0.3+0.1*5)--(0.5+1.8*0.5-0.3*3,0.3-1.8*0.3-0.3*5);
\tkzDefPoint(0.5,-1.73){dAB};
\node[above] at (0.5,-1.4) {$(AB)^*$};
\tkzDrawPoint(dAB);
\end{tikzpicture}
\caption{The domain $D$ is enclosed by semi-circles for which the sides of the triangle $T$ are diameters. For any $Z\in D$,  the triangle $T^Z$ is acute. The angle at the vertex dual to the edge $AB$ is $\pi-\angle AZB$.\label{fig:acuteDomain}}
\end{figure}
\begin{lemma}\label{lem:acutenessZones}
Let $T\subseteq\R^2$ be a triangle. 
For each of the sides of $T$, consider the semi-disk whose diameter is this side, and it is not disjoint from the interior of $T$.
Then for a point $Z$ in the interior of $T$, the triangle $T^Z$ is acute if and only if $Z$ is inside the intersection of those three semi-disks, see Figure \ref{fig:acuteDomain}.
\end{lemma}
\begin{proof}
It is well known that if the endpoints of a diameter of a disk are connected to a given point, then the resulting angle will be a right angle if the point is on the boundary circle, and will be an obtuse angle if the point is in the interior of the disk.
Thus, if $Z\in D$ then the angle $\angle AZB$ is obtuse.
By the construction of the dual polygon described in the paragraph above, the vertex dual to the edge $AB$, $(AB)^*$, is obtained by intersecting two lines that are perpendicular to $AZ$ and $ZB$.
This way we get a quadrilateral with two right angles, and by considering the sum of angles, we conclude that $\angle AZB$ is obtuse if and only if the angle at the vertex $(AB)^*$ is acute (see Figure \ref{fig:acuteDomain}).
The same reasoning applies to the other two angles of $T^Z$.
\end{proof}
\section{The ``Santal\'{o} point" of a convex body}\label{section:santalopt}
Our goal in this section is to prove Theorem \ref{thm:santaloPtIdentity}.
When discussing the volume product of a convex body $K$, if $z\not\in K$, then $K^z$ is unbounded, so its volume is infinite.
However, as explained in Subsection \ref{subsection:shortestBT}, unbounded domains can have closed billiard trajectories.
For this reason, a priori a point that minimizes $\alpha(K^z)$ can be any point in $\R^2$.
Thus, our first step will be to show that if $z\not\in K$, then $\alpha(K^z)=\infty$. 
\begin{lemma}\label{lem:dontLookOutside}
Let $K$ be a two-dimensional convex body, and $z\not\in K$. Then $\alpha(K^z)=\infty$. 
\end{lemma}
\begin{proof}
Let $z$ be any point not in $K$.
Since $K$ is a convex body, it follows that there are two supporting lines to $K$ that pass through $z$. 
This means that there exists a sector (the intersection of two half-planes determined by two non-parallel lines) $S_z$ with a vertex at $z$ for which $K\subseteq S_z$.
Passing to the dual gives the opposite inclusion, so we have $S_z^z \subseteq K^z$. 
It is immediate to check that the dual of a sector with respect to the vertex of the sector is again a sector.
For each $R>0$, let $B_R$ denote the disk of radius $R$ around the origin.
Then $K^z\cap B_R\supseteq S_z^z\cap B_R$. 
The right-hand side is a circular sector, of radius $R$, and the angle of the sector does not depend on $R$, see Figure \ref{fig:noBilliardInSector}.
All of these sectors are homothetic, which means that $\alpha(S_z^z \cap B_R) = cR$ for some constant $c>0$.
Since $\alpha$ is monotone with respect to inclusion, this means that $\alpha(K^z \cap B_R) \geq cR$, so $\set{\alpha(K^z\cap B_R)\mid R>0}$ is unbounded.
Now Lemma \ref{lem:unboundedImpossible} implies that $\alpha(K^z)=\infty$.
\end{proof}
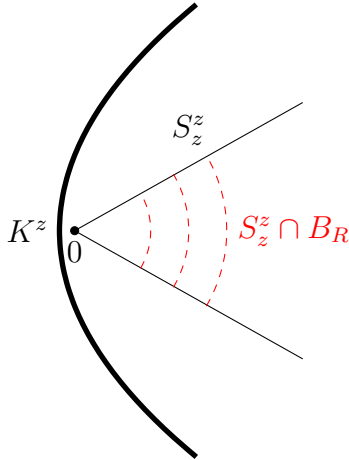
\begin{figure}
\centering
\begin{tikzpicture}[scale = 1]
\draw[smooth,black, line width = 2pt, domain = -3:3, variable  = \t] plot ({0.2*\t*\t-0.2},{\t});
\draw(0,0)--(3,1.7);
\draw(0,0)--(3,-1.7);
\tkzDefPoint(0,0){O};
\tkzDrawPoint(O);

\draw[dashed,red,domain=-30:30,variable = \t] plot({cos(\t)},{sin(\t)});
\draw[dashed,red,domain=-30:30,variable = \t] plot({1.5*cos(\t)},{1.5*sin(\t)});
\draw[dashed,red,domain=-30:30,variable = \t] plot({2*cos(\t)},{2*sin(\t)});

\node[left] at (-0.2,0) {$K^z$};
\node[below] at (0,0) {$0$};
\node[above] at (1.5,1) {$S_z^z$};
\node[red,right] at (2,0) {$S_z^z\cap B_R$};
\end{tikzpicture}
\caption{A domain that contains a sector cannot contain closed billiard orbits.\label{fig:noBilliardInSector}}
\end{figure}

In some cases it is simple to correspond the length of a shortest billiard trajectory of $K^z$ to a geometric quantity of $K$:
\begin{lemma}\label{lem:segmentsAndStrips}
Let $K$ be a convex body (of any dimension), and $z\in K$.
If $K$ is contained in a slab (the intersection of two half-spaces determined by two parallel hyperplanes) of width $w$, then $K^z$ contains a segment of length at least $4\slash w$.
Conversely, if $K$ contains a segment of length $L$ the midpoint of which is $z$, then $K^z$ is contained in a slab  of width $4\slash L$.
\end{lemma}
\begin{proof}
Without loss of generality we may assume that $z=0$ and the slab is vertical, of the form $S=\set{(x_1,...,x_n)\mid a-w\leq x_1 \leq a}$ for some $0<a<w$. 
Then the inclusion $K\subseteq S$ implies the opposite inclusion for the duals: $K^0\supseteq S^0$, see Figure \ref{fig:stripsAndSegments}.
 It is immediate to check that $S^0$ is the segment with end points $(\frac{1}{a-w},0,...,0)$ and $(\frac{1}{a},0,...,0)$.
This segment has length $\frac{1}{a}-\frac{1}{a-w}=\frac{w}{a(a-w)}$, which is at least $\frac{4}{w}$.

The second statement is proved similarly. 
Without loss of generality we may assume that $z=0$ and the segment is positioned along the $x_1$-axis. 
Then the endpoints of this segment are $(-\frac{L}{2},0,\dots,0)$, and $(\frac{L}{2},0,\dots,0)$. 
Denote this segment by $I$.
The inclusion $I\subseteq K$ implies the opposite inclusion for the duals: $K^0\subseteq I^0$. 
However, it is immediate to check that $I^0=\set{(x_1,...,x_n)\mid -\frac{2}{L}\leq x_1 \leq \frac{2}{L}}$, so this is a slab of width $\frac{4}{L}$. 

\end{proof}
\begin{figure}
\centering
\begin{subfigure}[b]{0.4\textwidth}
\centering
\begin{tikzpicture}[scale = 2]
\draw[smooth cycle,tension = 2,line width = 2pt] plot coordinates {(1,0)(0.8,0.3)(0.2,1)(0,0)(0.2,-1)(0.8,-0.3)(1,0)};
\draw[blue] (0,0)--(1,0);
\tkzDefPoint(0.7,0){A};
\tkzDrawPoint[blue](A);

\node[above] at (0.7,0) {$0$};
\node[below] at (0.8,0) {\scriptsize{$\frac{1}{a}$}};
\node[below] at (0.4,0) {\scriptsize{$\frac{1}{w-a}$}};
\node[above] at (0.2,1) {$K^z$};

\end{tikzpicture}
\caption{A segment of length at least $4\slash w$ is contained in $K^z$.}
\end{subfigure}
\hspace{1em}
\begin{subfigure}[b]{0.4\textwidth}
\centering
\begin{tikzpicture}[scale = 2]
\draw[smooth cycle, tension = 2, line width = 2pt] plot coordinates {(0,0)(0.2, 0.3)(0.8,1)(1,0)(0.8,-1)(0.2,-0.3)(0,0)};
\tkzDefPoint(0.3,0){A};
\tkzDrawPoint[blue](A);
\draw[dashed,blue](0,1)--(0,-1);
\draw[dashed,blue](1,1)--(1,-1);
\draw[dashed,blue](0.3,1)--(0.3,-1);
\draw[<->](0,0.4)--(0.3,0.4);
\draw[<->](0.3,0.4)--(1,0.4);

\node[above] at (0.05,0.4) {$w-a$};
\node[above] at (0.65,0.4) {$a$};

\node[below] at (0.4,0){$z$};
\draw[blue](0,0)--(1,0);

\node[above] at (0.8,1) {$K$};
\end{tikzpicture}
\caption{A strip of width $w$ contains $K$.}
\end{subfigure}
\caption{Correspondence of segments and slabs. \label{fig:stripsAndSegments}}
\end{figure}
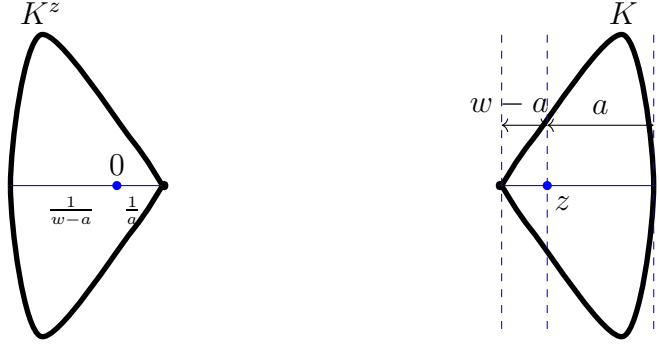
If we apply the second statement of Lemma \ref{lem:segmentsAndStrips} to the diameter of $K$, we see that if $z_0$ is the midpoint of a diameter of $K$, then $K^{z_0}$ is contained in a slab of width $4\slash\diam(K)$, and hence 
\begin{equation}\label{eq:midPointDiam}
\alpha(K^{z_0})\leq\frac{8}{\diam(K)}.
\end{equation}
This proves that $8\slash\diam(K)$ is an upper bound for $\inf\limits_{z\in K}\alpha(K^z)$, in any dimension.
In the two-dimensional case, we will now show that it is a lower bound, namely, that:
\begin{equation}\label{eq:alphaInequality}
\forall z\in K: \hspace{2mm} \alpha(K^z)\geq\frac{8}{\diam(K)}.
\end{equation}
Let $z\in K$.
By the results of \cite{BezdekDaniel2009Sbt}, there exists a shortest closed billiard trajectory of $K^z$ that is either a 2-orbit or a 3-orbit.
If it is a 2-orbit, then this orbit is a double normal, so $K^z$ is contained in a strip of width $\frac{\alpha(K^z)}{2}$.
Thus, by Lemma \ref{lem:segmentsAndStrips}, $(K^z)^0=K-z$ contains a segment of length at least $\frac{4}{\frac{\alpha(K^z)}{2}}=\frac{8}{\alpha(K^z)}$.
If $K-z$ contains such a segment, then so does $K$, and since the diameter is the length of the longest segment contained in $K$, we have $\frac{8}{\alpha(K^z)}\leq\diam(K)$, which gives \eqref{eq:alphaInequality} in this case (this proof  is again valid in all dimensions).
Now we wish to establish \eqref{eq:alphaInequality} for all $z$ such that $\alpha(K^z)$ is realized by a 3-orbit.
The proof relies on the following two reductions: from arbitrary $K$ to triangles, and from triangles to isosceles triangles.
In the case of isosceles triangles, we give a direct proof.
We begin with the reduction to triangles.
\begin{lemma}\label{lem:reductionToTriangles}
Suppose that for all triangles $T$, and for all $z\in T$ for which $T^z$ is an acute triangle we have $\alpha(T^z)\geq\frac{8}{\diam(T)}$.
Let $K$ be a convex body in $\R^2$, and let $z\in K$ be a point for which $\alpha(K^z)$ is realized by a 3-orbit.
Then $\alpha(K^z)\geq\frac{8}{\diam(K)}$.
\end{lemma}
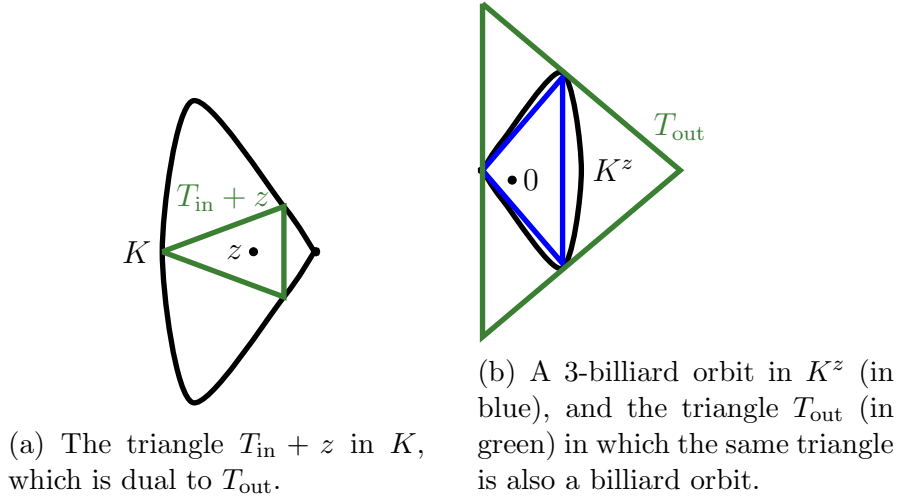
\begin{figure}
\centering
\begin{subfigure}[b]{0.4\textwidth}
\centering
\begin{tikzpicture}[scale = 2]
\draw[smooth cycle,tension = 2,line width = 2pt] plot coordinates {(1,0)(0.8,0.3)(0.2,1)(0,0)(0.2,-1)(0.8,-0.3)(1,0)};
\draw[OliveGreen,line width = 2pt](0,0)--(0.8,0.3)--(0.8,-0.3)--(0,0);
\tkzDefPoint(0.6,0){O};
\tkzDrawPoint(O);
\node[left] at (0.6,0) {$z$};
\node[left] at (0,0){$K$};
\node[OliveGreen,above] at (0.4,0.2) {$T_{\mathrm{in}}+z$};
\end{tikzpicture}
\caption{The triangle $T_{\mathrm{in}}+z$ in $K$, which is dual to $T_{\mathrm{out}}$.\label{fig:reductionToTrigOrig}}
\end{subfigure}
\hspace{1em}
\begin{subfigure}[b]{0.4\textwidth}
\begin{tikzpicture}[scale = 1.3]
\draw[smooth cycle, tension = 2, line width = 2pt] plot coordinates {(0,0)(0.2, 0.3)(0.8,1)(1,0)(0.8,-1)(0.2,-0.3)(0,0)};
\draw[blue,line width = 2pt] (0,0)--(0.81,0.95)--(0.81,-0.95)--(0,0);
\draw[OliveGreen,line width = 2pt](0,1.7)--(2,0)--(0,-1.7)--(0,1.7);
\tkzDefPoint(0.3,-0.1){O};
\tkzDrawPoint(O);
\node[right] at (0.3,-0.1){$0$};
\node[right] at (1,0) {$K^z$};
\node[OliveGreen,above] at (2,0.2) {$T_{\mathrm{out}}$};
\end{tikzpicture}
\caption{A 3-billiard orbit in $K^z$ (in blue), and the triangle $T_{\mathrm{out}}$ (in green) in which the same triangle is also a billiard orbit.\label{fig:reductionToTrigDual}}
\end{subfigure}
\caption{Reduction from arbitrary bodies to triangles.\label{fig:reductionToTriangles}}
\end{figure}
\begin{proof}
Assume that $z\in K$ is a point such that $\alpha(K^z)$ is realized by a 3-orbit.
Consider three support lines to $K^z$ at the bounce points of this orbit. 
These lines form a triangle $T_{\mathrm{out}}$ that contains $K^z$.
Observe that the billiard trajectory for $K^z$ is also a billiard trajectory for $T_{\mathrm{out}}$.
By Lemma \ref{lem:alphaTriangle}, if there is a billiard 3-orbit in a triangle, then this triangle is acute, and the 3-orbit is the shortest billiard orbit in it.
It then follows that $\alpha(K^z)=\alpha(T_{\mathrm{out}})$, see Figure \ref{fig:reductionToTrigDual}.
As a result, we get an opposite inclusion for the duals:
\[T_{\mathrm{out}}^0\subseteq (K^z)^0=K-z .\]
Write $T_{\mathrm{in}}=T_{\mathrm{out}}^0$, see Figure \ref{fig:reductionToTrigOrig}.
We then have $T_{\mathrm{in}}+z\subseteq K$, and it holds that $(T_{\mathrm{in}}+z)^z=T_{\mathrm{in}}^0=T_{\mathrm{out}}$.
Also, since $T_{\mathrm{in}}+z\subseteq K$, we must have \[\diam(T_{\mathrm{in}})\leq \diam(K).\]
Consequently,
\[\alpha(K^z)=\alpha(T_{\mathrm{out}})=\alpha((T_{\mathrm{in}}+z)^z).\]
By our assumption about triangles,\[\alpha((T_{\mathrm{in}}+z)^z)\geq\frac{8}{\diam(T_{\mathrm{in}})}\geq\frac{8}{\diam(K)}.\]
Combining the last two displayed lines we obtain the desired result, \[\alpha(K^z)\geq\frac{8}{\diam(K)}.\]
\end{proof}
Next we reduce the problem from arbitrary triangles to isosceles ones.
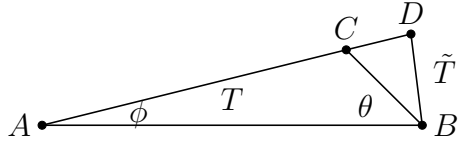
\begin{figure}
\centering
\begin{tikzpicture}[scale=5]
\tkzDefPoint(0,0){A};
\tkzDefPoint(1,0){B};
\tkzDefPoint(0.8,0.2){C};
\tkzDrawPoints(A,B,C);
\tkzDrawSegment(A,B);
\tkzDrawSegment(B,C);
\tkzDrawSegment(C,A);
\tkzDefPoint(0.9702,0.2419){D};
\tkzDrawPoint(D);
\tkzDrawSegment(C,D);
\tkzDrawSegment(D,B);
\node[left] at (0,0) {$A$};
\node[right] at (1,0){$B$};
\node[above] at (0.8,0.2){$C$};
\node[right] at (0.2,0.03) {$\phi$};
\node[left] at (0.9,0.05) {$\theta$};
\node at (0.5,0.07) {$T$};
\node[above] at (0.97,0.24) {$D$};
\node[right] at (1,0.15) {$\tilde{T}$};
\end{tikzpicture}
\caption{Reduction to isosceles triangle. The triangle $T$ is contained in the isosceles triangle $\tilde{T}$ which has the same diameter. The apex angle of $\tilde{T}$ is the smallest angle of $T$.\label{fig:reductionToIsosceles}}
\end{figure}
\begin{lemma}\label{lem:reductionToIsosceles}
Suppose that for all isosceles triangles $T$ with apex angle less than or equal to $\frac{\pi}{3}$, and for all $z\in T$ we have $\alpha(T^z)\geq\frac{8}{\diam(T)}$. 
Then this inequality holds for all triangles $T$ and for all $z\in T$.
\end{lemma}
\begin{proof}
Let $T$ be any triangle, and let $z\in T$. 
The diameter of $T$ is the longest edge of $T$.
Denote its length by $a$, and let $\theta\geq\phi$ be the angles of $T$ that are adjacent to this edge, see Figure \ref{fig:reductionToIsosceles}.
Since $\phi$ is the smallest angle, it follows that $\phi\leq\frac{\pi}{3}$.
The edge $AC$ near $\phi$ is not longer than the edge $AB$, and we extend it until it will have length $a$.
Call the newly obtained isosceles triangle $\tilde{T}$.
Then $\tilde{T}$ is an isosceles triangle with apex angle at most $\frac{\pi}{3}$, and $\diam(\tilde{T})=\diam(T)=a$. 
Since $T\subseteq \tilde{T}$, we have $\tilde{T}^z\subseteq T^z$, and hence $\alpha(T^z)\geq\alpha(\tilde{T}^z)$. 
Since $\tilde{T}$ is an isosceles triangle, we can use our assumption, and conclude that
\[\alpha(T^z)\geq\alpha(\tilde{T}^z)\geq\frac{8}{\diam(\tilde{T})}=\frac{8}{\diam(T)}.\]
\end{proof}
Finally, we present the proof for isosceles triangles.
Since both sides of the inequality \eqref{eq:alphaInequality} are $-1$-homogeneous (in $K$), we can assume for simplicity that $\diam(K)=1$.
Also, recall that the inequality $\alpha(K^z)\geq \frac{8}{\diam (K)}$ was already shown for all $z$ for which $\alpha(K^z)$ is realized by a 2-orbit, so we can only check the points $z$ for which $\alpha(K^z)$ is realized by a 3-orbit.
\begin{lemma}\label{lem:proofForIsosceles}
Let $K$ be an isosceles triangle for which the apex angle is at most $\frac{\pi}{3}$, and for which the diameter is equal to $1$ (in this case, the diameter is realized by the legs of the triangle). 
Then for any $Z\in K$ for which $\alpha(K^Z)$ is realized by a 3-orbit, we have $\alpha(K^Z)\geq 8$.
\end{lemma}
\begin{proof}
Let $K$ be such a triangle, and let $Z\in K$ be a point for which $K^Z$ is acute. 
Denote the vertices of the triangle $K$ by $A$, $B$, $C$, and denote its sides by $a$, $b$, $c$, where $a$ is the side opposite to $A$ and similarly for the others.
Call the vertices of $K^z$ $A^*$, $B^*$, $C^*$, where $A^*$ is the vertex dual to the edge $a$, and similarly for the other vertices.
By Lemma \ref{lem:alphaTriangle}, $\alpha(K^Z)=2h_{A^*}\sin\delta$, where $h_{A^*}$ is the length of the altitude of $K^Z$ dropped from $A^*$, and $\delta$ is the angle at the vertex $A^*$.
We now interpret this quantity in terms of the original triangle $K$. 
\begin{figure}
\centering
\begin{tikzpicture}[scale = 5]
\draw[black, line width = 1.5pt] (0.2588,0)--(-0.2588,0)--(0,0.9659)--(0.2588,0);
\tkzDefPoint(0.1294,0.1609){Z};
\tkzDefPoint(0.2588,0){B};
\tkzDefPoint(-0.2588,0){C};
\tkzDefPoint(0,0.9659){A};
\tkzDefPoint(0.1294,-0.3566){Astar};
\tkzDefPoint(0.1294+0.6*0.9695,0.1609+0.6*0.2588){Bstar};
\tkzDefPoint(0.1294-0.2*0.9695,0.1609+0.2*0.2588){Cstar};
\tkzDrawPoints(Astar,Bstar,Cstar);
\tkzDrawSegment[blue,dashed,line width = 1.5pt](Astar,Bstar);
\tkzDrawSegment[blue,dashed,line width = 1.5pt](Bstar,Cstar);
\tkzDrawSegment[blue,dashed,line width = 1.5pt](Cstar,Astar);
\tkzDrawPoints(A,B,C);
\tkzDrawPoints(Z);

\tkzDefPoint(0+1.2*0.1294,0.9659-1.2*0.8050){D};
\tkzDrawPoints(D);
\tkzDrawSegment[green](A,D);
\tkzDefPoint(0+0.9*0.1294,0.9659-0.9*0.8050){E};
\tkzDrawPoint(E);
\tkzDefPoint(0+1.63*0.1294,0.9659-1.63*0.8050){F};
\tkzDrawPoint(F);
\draw[red, dashed](0.1294-0.3*0.8050,-0.3566-0.3*0.1294)--(0.1294+1*0.8050,-0.3566+1*0.1294);
\tkzDrawSegment[green](E,F);
\tkzDefPoint(0.1294-0.73*0.1294,-0.3566+0.73*0.8050){G};
\tkzDrawSegment[red,dashed](Astar,G);
\tkzDrawPoint(G);

\tkzMarkRightAngle[size = 0.05](Astar,G,Cstar);
\tkzMarkRightAngle[size = 0.05](F,E,Bstar);

\node[above] at (0,0.9659){$A$};
\node[left] at (-0.2588,0){$C$};
\node[right] at (0.2588,0){$B$};

\node[below] at (0.1294,-0.3566){$A^*$};
\node[right] at (0.1294+0.6*0.9695,0.1609+0.6*0.2588) {$C^*$};
\node[left] at (0.1294-0.2*0.9695,0.1609+0.2*0.2588) {$B^*$};
\node at (0.1294+0.05,0.1609){$Z$};
\node at (0+1.2*0.1294+0.05,0.9659-1.2*0.8050-0.05){$D$};
\node at (0+0.9*0.1294+0.03,0.9659-0.9*0.8050+0.05){$E$};
\node[below] at (0+1.63*0.1294,0.9659-1.63*0.8050){$F$};
\node[above] at (0.1294-0.73*0.1294,-0.3566+0.73*0.8050){$G$};
\end{tikzpicture}
\caption{Interpreting the shortest billiard trajectory of $K^z$ in terms of $K$.\label{fig:fromDualToOrig}}
\end{figure}
Referring to Figure \ref{fig:fromDualToOrig}, we get $\alpha(K^Z)=2|A^* G|\sin\angle A^*$.
From the construction of $K^Z$ it follows that $\angle BZC=\pi-\angle A^*$, so $\alpha(K^z)=2|A^* G|\sin\angle BZC$.
The length of the segment $A^* G$ is the same as the length of the segment $EF$ parallel to it that passes through $Z$. 
By the construction of the dual triangle, it holds that $|EZ|=\frac{1}{|AZ|}$.
We explain why $|ZF|=\frac{1}{|ZD|}$.
The line $A^*F$ is dual to some point $D'$ on the line containing $ZD$, and $|ZF|=\frac{1}{|ZD'|}$, so it is enough to explain why $D=D'$.
By the construction of the dual, every point of a given line $\ell$ is dual to some line, and $\ell$ itself is dual to the intersection points of all these dual lines. 
Consequently, the point $D'$ is the intersection of all the lines  that are the duals to points on the line containing $A^* F$. 
The line dual to $A^*$ itself is the line containing $BC$, so the point $D'$ must be the intersection point of $BC$ with the line containing $ZD$, and therefore $D=D'$, and $|ZF|=\frac{1}{|ZD|}$.
Overall we get the formula
\begin{equation}\label{eq:geometricFormulaAlphaKz}
\alpha(K^Z)=2\Big(\frac{1}{|AZ|}+\frac{1}{|ZD|}\Big)\sin\angle BZC.
\end{equation}
In this formula we have ``eliminated" the dual, and this expression is now ``intrinsic" in terms of $K$.
Now position the triangle $ABC$ in the complex plane such that the origin is the midpoint of $BC$, and $BC$ lies on the real axis. 
Let $\phi\geq\frac{\pi}{3}$ be the base angle of this triangle. 
Then $B=\cos\phi$, $A=i\sin\phi$, and we write $z$ for the complex coordinate of the point $Z$.
First we compute the length of $AD$. 
Use the sine theorem in the triangle $ACD$.
If $\theta$ is the angle between $AC$ and $AD$, then $|AD|=\frac{\sin\phi}{\sin(\phi+\theta)}$.
On the other hand, by elementary geometry it follows that $\theta=\arg(z-i\sin\phi )-(\phi+\pi)$.
Combining the two, we get that $|AD|=-\frac{\sin\phi|z-i\sin\phi |}{\im(z-i\sin\phi)}$.
Also, we have $|AZ|=|z-i\sin\phi|$, and 
\begin{gather*}|ZD|=|AD|-|AZ|=-|z-i\sin\phi|(1+\frac{\sin\phi}{\im(z-i\sin\phi)})=\\
=-|z-i\sin\phi|\frac{\im z}{\im(z-i\sin\phi)}.\end{gather*}
By the definition of the point $Z$, it does not lie between $B$ and $C$, so $\im z > 0$.
To compute $\sin\angle BZC$, we compute the area of the triangle $BZC$ in two different ways. 
On the one hand, this area is equal to \[\frac{1}{2}|z-\cos\phi||z+\cos\phi|\sin\angle BZC,\] and on the other hand it is equal to $\frac{1}{2}\im z\cdot 2\cos\phi=\im z\cos\phi$.
This implies that $\sin\angle BZC=\frac{2\cos\phi\im z}{|z^2-\cos^2\phi|}$.
Now we can compute $\alpha(K^z)$:
\begin{gather*}\alpha(K^z)=2\frac{1}{|z-i\sin\phi|}\Big(1-\frac{\im(z-i\sin\phi)}{\im z}\Big)\cdot \frac{2\cos\phi\,\im z}{|z^2-\cos^2\phi|}=\\
=\frac{4\cos\phi}{|z^2-\cos^2\phi||z-i\sin\phi|}\cdot\frac{\im z-\im(z-i\sin\phi)}{\im z}\cdot\im z=\\
=\frac{2\sin 2\phi}{|z^2-\cos^2\phi||z-i\sin\phi|}.\end{gather*}
Therefore, we arrived to the following conclusion: if $a,b,c$ are the vertices of an isosceles triangle $K$  with base angle $\phi\geq\frac{\pi}{3}$ and diameter $1$, and $z$ is a point for which $\alpha(K^z)$ is a 3-orbit, then 
\begin{equation}\label{eq:elegantFormula}
\alpha(K^z)=\frac{2\sin(2\phi)}{|(z-a)(z-b)(z-c)|}.
\end{equation}
By Lemma \ref{lem:acutenessZones}, for $K^z$ to be an acute triangle, the point $z$ must be inside the domain $H$ described in Figure \ref{fig:minimizationDomain}.
\begin{figure}
\centering
\begin{tikzpicture}[scale = 3]
\draw[black, line width = 1.5pt] (0.2588,0)--(-0.2588,0)--(0,0.9659)--(0.2588,0);
\tkzDefPoint(0.1294,0.1609){Z};
\tkzDefPoint(0.2588,0){B};
\tkzDefPoint(-0.2588,0){C};
\tkzDefPoint(0,0.9659){A};
\draw[domain = 30:150, smooth, variable = \t, blue] plot({0.2588*cos(\t)},{0.2588*sin(\t)});
\draw[domain = 225:255, smooth, variable = \t, blue] plot({0.1294+0.5*cos(\t)},{0.4829+0.5*sin(\t)});
\draw[domain = -75:-45, smooth, variable = \t, blue] plot({-0.1294+0.5*cos(\t)},{0.4829+0.5*sin(\t)});
\node at (0,0.1) {$H$};
\end{tikzpicture}
\caption{Domain in which $K^z$ is acute. \label{fig:minimizationDomain}}
\end{figure}
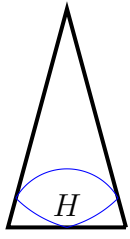
The function $\alpha(K^z)$ is the absolute value of a complex rational function.
It is well known that the modulus of a non-constant non-vanishing holomorphic function does not have local extreme points.
 Thus, the minimum of $\alpha(K^z)$ over $z\in H$ must be attained for $z\in \bd H$.
If $z$ is on the upper circle arc, then $\angle BZC=\frac{\pi}{2}$, and then $\alpha(K^z)=2(\frac{1}{|AZ|}+\frac{1}{|ZD|})$. 
It is readily verified that $\frac{1}{|AZ|}+\frac{1}{|ZD|}\geq\frac{4}{|AD|}$, so $\alpha(K^z)\geq \frac{8}{|AD|}$.
We also have $|AD|\leq 1$ since $1$ is the diameter of the triangle, and therefore $\alpha(K^z)\geq 8$.
Similarly, if $z$ is on the right or left arcs (they are symmetric, so suppose it is on the right arc), then we use the fact that equation \eqref{eq:geometricFormulaAlphaKz} does not depend on the vertex we connect $z$ to (because the formula in Lemma \ref{lem:alphaTriangle} does not depend on the choice of altitude), so we can also write:
\[\alpha(K^z)=2\Big(\frac{1}{|BZ|}+\frac{1}{|ZD'|}\Big)\sin\angle AZC,\]
where $D'$ is the intersection point of $BZ$ and $AC$.
But if $z$ is on the right arc, then $\angle AZC=\frac{\pi}{2}$, so by a similar argument $\alpha(K^z)\geq \frac{8}{|BD'|}\geq 8$.
\end{proof}
The proof of Theorem \ref{thm:santaloPtIdentity} is now complete: whether $\alpha(K^z)$ is realized by a 2-orbit or a 3-orbit, we established inequality \eqref{eq:alphaInequality}, and together with the inequality \eqref{eq:midPointDiam}, and Lemma \ref{lem:dontLookOutside}, the result of Theorem \ref{thm:santaloPtIdentity} is obtained.
\section{Upper bounds for the billiard product}\label{section:polygonOptimization}
In view of Theorem \ref{thm:santaloPtIdentity}, we can simplify the billiard product in the plane to the quantity
\begin{equation}\label{eq:simplerBeta}
\beta(K)=\inf\limits_{z\in\R^2}\alpha(K)\alpha(K^z)=\alpha(K)\inf\limits_{z\in K} \alpha(K^z)=\frac{8\alpha(K)}{\diam(K)}.
\end{equation}
Observe that from Lemma \ref{lem:alphaCont} and the continuity of the diameter it follows that $K\mapsto \beta(K)$ is a continuous function with respect to the Hausdorff topology.
\subsection{Global upper bound for the billiard product}\label{subsec:globalUpperBound}
In this subsection we prove Theorem \ref{thm:santaloAnalogue}, and in the planar case we identify all of the maximizers of the billiard product. 
\begin{proof}[Proof of Theorem \ref{thm:santaloAnalogue}]
As we saw in Section \ref{section:santalopt}, it follows from Lemma \ref{lem:segmentsAndStrips} that
\[\inf_{z\in K}\alpha(K^z)\leq\frac{8}{\diam(K)}.\]
Any diameter of $K$ is a double normal to $K$, so it is a generalized billiard orbit.
Hence, $\alpha(K)\leq 2\,\diam(K)$.
Consequently,
\[\beta(K)=\alpha(K)\inf\limits_{z\in K} \alpha(K^z)\leq 2\,\diam(K)\cdot\frac{8}{\diam(K)} =16.\]
We claim that $\beta(B)=16$ for any ball $B\subseteq\R^n$.
Indeed, if $B$ is a unit ball, then its shortest billiard trajectory is along a diameter, and $\alpha(B)=4$.
If $z\in B$ is at distance $0<r<1$ from the center, then $B^z$ is an ellipsoid of revolution with the semi-axes $\frac{1}{1-r^2}$ and $\sqrt{\frac{1}{1-r^2}}$.
It is know (see, e.g., \cite{GhomiMohammad2004Spbt}) that for an ellipsoid, any shortest billiard trajectory is along the shortest axis, and hence $\alpha(B^z)> 4$.
We see that indeed $\alpha(B)\alpha(B^z)\geq 16$ for all $z\in B$, and the value 16 is clearly attained when $z$ is the center.
Now for the actual equality $\beta(K)=16$ to occur we must have $\alpha(K)=2\,\diam(K)$.
However, any segment of minimal width of $K$ is always a generalized billiard trajectory (it is a double normal), so $\alpha(K)\leq 2 w(K)$, where $w(K)$ denotes the minimal width of $K$.
Hence $\diam(K)\leq w(K)$, but since the diameter is the maximal width, then $\diam(K)=w(K)$ and thus $K$ has constant width.
\end{proof}
In the planar case we can show the converse statement: if $K$ has constant width, then $\beta(K)=16$.
Indeed, by Theorem 4.1 of \cite{BalitskiyAlexey2016Scbt}, in a two-dimensional convex body of constant width, the shortest closed billiard trajectories are 2-orbits. 
The shortest 2-orbit inside a convex body is its minimal width, so if $K$ has constant width, then $\alpha(K)=2w(K)=2\,\diam(K)$, and thus \eqref{eq:simplerBeta} implies that $\beta(K)=16$.

%
\subsection{The billiard product for polygons}\label{subsec:billiardPolygon}

We now try to maximize $\beta(K)$ over some subclasses of convex bodies in $\R^2$, for example, for polygons with at most $n$ vertices. 
Let $\mathcal{P}_n$ denote the set of planar polygons with at most $n$ vertices.
We wish to investigate $\beta_n :=\sup\limits_{P\in \mathcal{P}_n} \beta(P)$.
By equation \eqref{eq:simplerBeta}, maximizing $\beta$ over $\mathcal{P}_n$ is the same as maximizing $\alpha$ over the class of polygons with diameter $1$, and at most $n$ vertices.
Let us begin with an example.
Let $R_n$ be the regular polygon with $n$ vertices. 
The values of $\alpha(R_n)$, assuming $R_n$ is inscribed inside a circle of radius $1$, were computed, e.g. in \cite[Theorem 4.2]{AlkoumiNaeem2015Scbo}:
\[\alpha(R_n)=\begin{cases}
\frac{3\sqrt{3}}{2}, \textrm{ if } n=3 , \\
2(1+\cos\frac{\pi}{n})=4(\cos\frac{\pi}{2n})^2, \textrm{ if } n\geq 5, \textrm{ odd} , \\
4\cos\frac{\pi}{n}, \textrm{ if } n  \textrm{ even} . \\
\end{cases}\]
If $n$ is even, then the diameter of $R_n$ is $2$, and if $n$ is odd, then the diameter of $R_n$ is the longest diagonal  which has the length $2\cos\frac{\pi}{2n}$ (for $n=3$ this is the length of the edge). 
Thus we get the following values:
\[\beta(R_n)=\begin{cases}
12, \textrm{ if } n = 3, \\
16\cos\frac{\pi}{2n}, \textrm{ if } n\geq 5\textrm{ odd},\\
16\cos\frac{\pi}{n}, \textrm{ if } n \textrm{ even} .
\end{cases}\]If 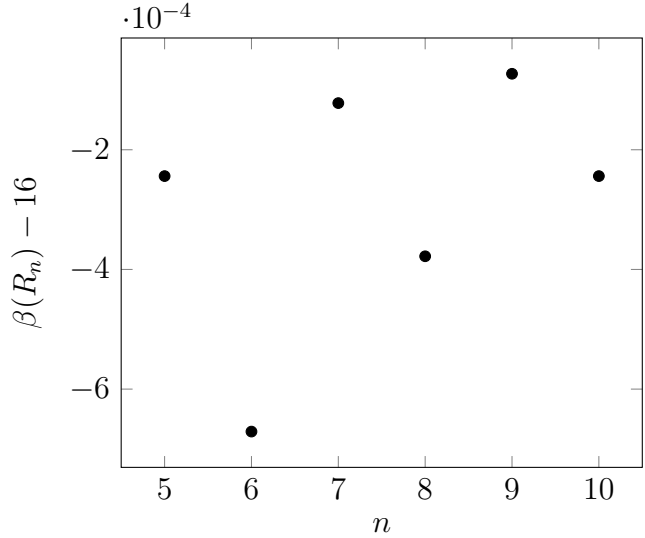
\begin{figure}
\centering
\begin{tikzpicture}
\begin{axis}[xlabel = {$n$}, ylabel = {$\beta(R_n)-16$}]
\addplot[only marks] coordinates {(5,1.9999699*8-16)(6,1.999916*8-16)(7,1.9999846*8-16)(8,1.999953*8-16)(9,1.9999907*8-16)(10,1.9999699*8-16)};
\end{axis}
\end{tikzpicture}
\caption{Behavior of the residuals $\beta(R_n)-16$. \label{fig:betaRn}}
\end{figure} we consider separately odd and even values of $n$, then each of these sequences is monotone, but the value for each odd $n$ is larger than the value for the succeeding even number, so we have $\beta(R_3)>\beta(R_4)$, $\beta(R_5)>\beta(R_6)$, and so on, see Figure \ref{fig:betaRn}.
As shown in Lemma \ref{lem:alphaCont}, $K\mapsto\alpha(K)$ is continuous in the Hausdorff topology. 
Therefore, if we consider a regular $n$-gon with odd $n$, and truncate it arbitrarily close to one of its vertices, we will get an $(n+1)$-gon $K$ with the same diameter as $R_n$ for which $\beta(K)>\beta(R_{n+1})$. 
This means that $\beta_n \geq \max\set{\beta(R_n),\beta(R_{n-1})}$.
In particular, the regular polygons with even number $n$ of vertices do not maximize $\beta$ in the class of polygons with $\leq n$ vertices.
We now prove Theorem \ref{thm:optimalTriangle}, which says that $\beta_3=\beta(R_3)=12$.
\begin{proof}[Proof of Theorem \ref{thm:optimalTriangle}]
Let $T$ be a triangle of diameter $1$ (i.e., its longest edge has length $1$).
Let $\theta$ and $\phi$ be the angles near the longest edge, with $\theta\geq \phi$.
Continue the edge that is near the angle $\phi$ until its length becomes $1$.
This way we get an isosceles triangle $\tilde{T}$ with leg length $1$ and apex angle $\phi$, see Figure \ref{fig:reductionToIsosceles}. 
Using the formula from Lemma \ref{lem:alphaTriangle}, we have $\beta(T)=8\alpha(T)\leq 8\alpha(\tilde{T})=8\cdot(2\cos\frac{\phi}{2}\sin\phi)=32\cos^2\frac{\phi}{2}\sin\frac{\phi}{2}$.
Since $\phi$ is the smallest angle of $T$, we must have $\phi\leq\frac{\pi}{3}$. 
If we write $t=\sin\frac{\phi}{2}$, then $\alpha(\tilde{T})$ is a cubic polynomial in $t$. 
Since $0\leq\phi\leq\frac{\pi}{3}$, the range for $t$ is $0\leq t\leq\frac{1}{2}$. 
It is easily checked that this polynomial increases in this range, and thus the maximal value is attained when $\phi=\frac{\pi}{3}$, which means that $T$ is a regular triangle, and then
\[\beta(T)\leq 32\cdot \frac{3}{4}\cdot\frac{1}{2}=12.\]
\end{proof}
We next turn to discuss other polygons. 
We present a quadrilateral which we conjecture maximizes $\beta$ for quadrilaterals, along with some heuristics as to why this should be true.
 Consider the quadrilateral $K$ with the vertices $(\pm\frac{1}{2},0)$,$(0,\frac{\sqrt{3}}{2})$,$(0,\frac{\sqrt{3}}{2}-1)$, see Figure \ref{fig:theOptimalQuad}.
\begin{figure}
\centering
\begin{tikzpicture}[scale = 2]
\tkzDefPoint(0.5,0){A};
\tkzDefPoint(-0.5,0){B};
\tkzDefPoint(0,0.866){C};
\tkzDefPoint(0,-0.1339){D};
\tkzDefPoint(-0.3415,0.2745){E};
\tkzDefPoint(0.25,0.433){F};
\tkzDefPoint(-0.25,-0.0669){G};
\tkzDrawPoints(A,B,C,D);
\tkzDrawPoints[blue](E,F,G);
\tkzDrawSegment[line width = 1.5pt](A,C);
\tkzDrawSegment[line width = 1.5pt](C,B);
\tkzDrawSegment[line width = 1.5pt](B,D);
\tkzDrawSegment[line width = 1.5pt](D,A);
\tkzDrawSegment[blue](E,F);
\tkzDrawSegment[blue](F,G);
\tkzDrawSegment[blue](G,E);
\end{tikzpicture}
\caption{The conjectured quadrilateral that maximizes $\beta$, and its shortest billiard trajectory (in blue). \label{fig:theOptimalQuad}}
\end{figure}
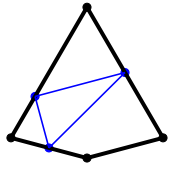
This quadrilateral is obtained by taking a regular triangle with side length $1$, and extending one of its altitudes until it will have length one. 
The result is a quadrilateral that contains the regular triangle, but still has diameter $1$, so this operation increases the billiard product.
We compute the length of its shortest billiard trajectory (using the method of \cite{AlkoumiNaeem2015Scbo}, and the formula from Lemma \ref{lem:alphaTriangle}), and get $\alpha(K)=\sqrt{3}\cos\frac{\pi}{12}\approx 1.67$, so $\beta(K)=8\alpha(K)\approx 13.3>12$.
Therefore this is a quadrilateral for which $\beta(K)>\beta_3$, which shows that $\beta_4>\beta_3$. 
We conjecture that this is a maximizer of $\beta$ in $\mathcal{P}_4$.

As promised, here are some heuristic arguments that support the maximality of this polygon.
Suppose $Q$ is any quadrilateral with diameter $1$.
Its diameter can be attained either at an edge, or at a diagonal.
First we can consider the case when an edge has length $1$.
Position $Q$ so that the vertices of the edge $AB$ of length $1$ are $(\pm\frac{1}{2},0)$.
The other two vertices, $C$, $D$, must lie on the same side of the $x$-axis, and inside the disks $(x\pm\frac{1}{2})^2+y^2\leq 1$, see Figure \ref{fig:edgeCase}.
\begin{figure}
\centering
\begin{subfigure}[b]{0.4\textwidth}
\centering
\begin{tikzpicture}[scale = 2]
\tkzDefPoint(-0.5,0){A};
\tkzDefPoint(0.5,0){B};
\tkzDrawPoints(A,B);
\tkzDrawSegment(A,B);
\draw[domain = 0:60, smooth, variable = \t] plot ({-0.5+cos(\t)},{sin(\t)});
\draw[domain = 120:180, smooth, variable = \t] plot ({0.5+cos(\t)},{sin(\t)});
\tkzDefPoint(0.25,0.2){C};
\tkzDefPoint(-0.3,0.3){D};
\tkzDrawPoints(C,D);
\tkzDefPoint(0.4662,0.2576){CC};
\tkzDefPoint(-0.4363,0.3511){DD};
\tkzDrawPoints(CC,DD);
\tkzDrawSegment(B,C);
\tkzDrawSegment(C,D);
\tkzDrawSegment(D,A);
\tkzDrawSegment[dashed,blue](B,CC);
\tkzDrawSegment[dashed,blue](CC,DD);
\tkzDrawSegment[dashed,blue](DD,A);

\node[below] at (-0.5,0) {$A$};
\node[below] at (0.5,0) {$B$};
\node[below] at (0.25,0.2) {$C$};
\node[below] at (-0.3,0.3) {$D$};
\node[right] at (0.46,0.25) {$C'$};
\node[left] at (-0.43, 0.35){$D'$};
\end{tikzpicture}
\caption{The diameter is realized by an edge. \label{fig:edgeCase}}
\end{subfigure}
\hspace{1em}
\begin{subfigure}[b]{0.4\textwidth}
\centering
\begin{tikzpicture}[scale = 2]
\tkzDefPoint(-0.5,0){A};
\tkzDefPoint(0.5,0){C};
\draw[domain = -60:60, smooth, variable = \t] plot ({-0.5+cos(\t)},{sin(\t)});
\draw[domain = 120:240, smooth, variable = \t] plot ({0.5+cos(\t)},{sin(\t)});
\tkzDefPoint(0.35,-0.3){B};
\tkzDefPoint(-0.3,0.3){D};
\tkzDrawPoints(C,D);
\tkzDrawPoints(A,B);
\tkzDrawSegment(B,C);
\tkzDrawSegment(C,D);
\tkzDrawSegment(D,A);
\tkzDrawSegment(A,B);


\node[left] at (-0.5,0) {$A$};
\node[right] at (0.5,0) {$C$};
\node[below] at (0.3,-0.3) {$B$};
\node[above] at (-0.3,0.3) {$D$};
\tkzDrawSegment[dashed](B,D);
\node[right] at (0,0) {$1$};
\end{tikzpicture}
\caption{The diameter is realized by a diagonal. \label{fig:diagonalCase}}
\end{subfigure}
\caption{Optimizing $\beta$ for quadrilaterals.\label{fig:optimalQuad}}
\end{figure}
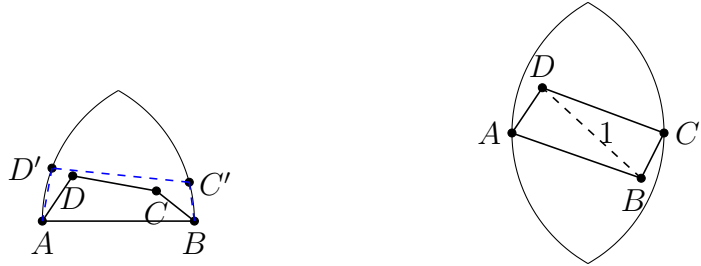
Note that the distance between any two points in the domain above the $x$ axis enclosed by these two circle arcs is at most $1$. 
Hence, we can always extend the diagonals $AC$ and $BD$ to segments $AC'$ and $BD'$ that will have length $1$, or until both points reach the boundary arc of circle (it is also possible that the points $C'$ and $D'$ will be on the same arc of circle, and then one of the diagonal lengths will be shorter than $1$). 
The obtained quadrilateral, $\tilde{Q}$, contains $Q$, and still has diameter $1$, and hence $\beta(Q)\leq\beta(\tilde{Q})$.
One can compute the value of $\alpha(\tilde{Q})$ as an explicit function of the angles $\angle CAB$, $\angle DBA$. 
In some cases it is possible to prove the inequality $\alpha(\tilde{Q})\leq\sqrt{3}\cos(\pi\slash 12)$ analytically, but in some cases we only have strong numerical evidence for this inequality.

The other case is where a diagonal of $Q$ has length $1$. 
If we call the vertices of this diagonal $A$ and $C$, and position them at $(\pm\frac{1}{2},0)$, then the remaining vertices need to be in the domain $(x\pm\frac{1}{2})^2+y^2 \leq 1$, and must lie on opposite sides of the $x$-axis, see Figure \ref{fig:diagonalCase}.
We can again try to extend the segments $AB$ and $CD$ until they will either have length $1$, or the points $B$ and $D$ will be on the boundary arcs of circle, but it may be possible that the length of $BD$ will become $1$ before either of these conditions is fulfilled (for example, if $ABCD$ are the vertices of a square with diagonal $1$). 
Hence, the space of such quadrilaterals depends on three parameters, and direct analysis of this situation is even more complicated.
However, even in this case, numerical analysis still supports the inequality $\alpha(Q)\leq \sqrt{3}\cos(\pi\slash 12)$.
\subsection{Steiner symmetrization and the billiard product}\label{subsec:steinerSymm}
In this subsection we discuss how the Steiner symmetrization affects the billiard product.
First, we recall the definition:
Given a convex body $K\subseteq\R^n$ and a hyperplane $H$, consider all the lines orthogonal to $H$ that have a non-empty intersection with $K$ (this intersection can be either a point or a segment).
The Steiner symmetrization of $K$ with respect to $H$, denoted by $S_H(K)$, is obtained by shifting all those segments to have their midpoint on $H$.
It is known that $\diam(K)\geq\diam(S_H(K))$.
For more details on the Steiner symmetrization, see e.g., \cite{EgglestonHaroldGordon1969C, Artstein-AvidanShiri2015AGAP}.

In general, Steiner symmetrization can decrease the billiard product, as we demonstrate in the following example. 
Let $K$ be the triangle with vertices $(\pm 1,0)$,$(0,1)$. 
Then $K$ is right angled, and $\alpha(K)=2$.
The diameter of $K$ is also $2$, so from \eqref{eq:simplerBeta}, $\beta(K)=8$.
Let $\l$ be the $x$-axis. 
One can check that $S_\l (K)$ is the rhombus with vertices $(\pm 1,0)$, $(0,\pm \frac{1}{2})$.
The closed billiard trajectories in this rhombus are the altitudes from the vertices $(0,\pm\frac{1}{2})$ to the opposite edges, and their lengths are equal to $2\cdot\frac{2}{\sqrt{5}}<2$, and thus $\alpha(S_\l(K))<\alpha(K)$. 
The diameter of $S_\l(K)$ is still $2$, so 
\[\beta(S_\l(K))=\frac{8\alpha(S_\l(K))}{\diam(S_\l(K))}<\frac{8\alpha(K)}{\diam(K)}=\beta(K),\]
and we see that this symmetrization decreased the billiard product.
However, the next claim shows that in any triangle we can find a direction in which the Steiner symmetrization increases the billiard product.
\begin{claim}\label{claim:steinerBettaTriangle}
Let $T$ be a triangle, and $\l$ be one of its altitudes. 
Denote by $S_\l$ the Steiner symmetrization with respect to $\l$.
Then 
\[\beta(T)\leq \beta(S_\l (T)).\]
\end{claim}
If it is possible to generalize this claim to other bodies, then it might also allow to investigate maximizers of $\beta$ by means of symmetrization.
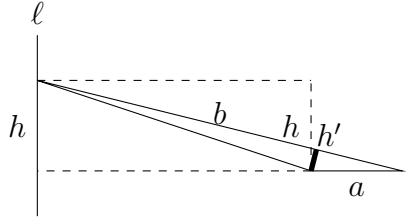
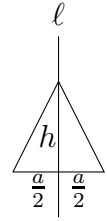
\begin{figure}
\centering
\begin{subfigure}[b]{0.4\textwidth}
\centering
\begin{tikzpicture}[scale = 1.2]
\draw(0,0)--(1,0)--(-3,1)--(0,0);
\draw[dashed](0,0)--(-3,0);
\draw(-3,-0.5)--(-3,1.5);
\draw[dashed](0,0)--(0,1);
\draw[line width = 2pt](0,0)--(0.0588,0.2352);
\draw[dashed](-3,1)--(0,1);

\node[below] at (0.5,0){$a$};
\node[left] at (-3,0.5) {$h$};
\node[left] at (0,0.5) {$h$};
\node[above] at (-1,0.4) {$b$};
\node at (0.2,0.4) {$h'$};
\node[above] at (-3,1.5) {$\l$};
\end{tikzpicture}
\caption{The obtuse triangle $T$.
The altitude $h'$ is shorter than the altitude $h$. \label{fig:obtuseSteinerLeft}}
\end{subfigure}
\hspace{1em}
\begin{subfigure}[b]{0.4\textwidth}
\centering
\begin{tikzpicture}[scale = 1.2]
\draw(0,0)--(1,0)--(0.5,1)--(0,0);
\draw(0.5,-0.5)--(0.5,1.5);

\node[left] at (0.5,-0.2) {$\frac{a}{2}$};
\node[right] at (0.5,-0.2) {$\frac{a}{2}$};
\node[above] at (0.5,1.5) {$\l$};
\node[left] at (0.6,0.4) {$h$};
\end{tikzpicture}
\caption{The Steiner symmetrization $S_\l(T)$ is acute. \label{fig:obtuseSteinerRight}}
\end{subfigure}
\caption{Steiner symmetrization for an obtuse triangle.\label{fig:obtuseSteiner}}
\end{figure}
\begin{proof}
First observe that the symmetrization of a triangle with respect to an altitude always gives an isosceles triangle.
The proof examines separately the following cases: whether $T$ or $S_\l(T)$ are acute or not, and whether $\l$ is an inner or outer altitude in $T$.
For simplicity, when we write below ``obtuse" we mean ``either obtuse or right angled".
Also, in view of \eqref{eq:simplerBeta} and the fact that Steiner symmetrization does not increase the diameter, it is enough to show that in all the listed cases $\alpha(T)\leq\alpha(S_\l(T))$.\\
\textbf{Case 1:} $\l$ is an outer altitude.
In this case $T$ cannot be acute, so we have only two cases to check, according to whether $S_\l (T)$ is acute or not.\\
\textbf{Case 1a:} the triangle $S_\l(T)$ is obtuse.
Denote by $h'$ the length of the inner altitude in $T$, and by $h$ the length of the altitude $\l$ in $T$.
By the definition of the symmetrization, it follows that $\l$ is now an inner altitude in $S_\l (T)$, and the length of that altitude is still $h$. 
Since both $T$ and $S_\l (T)$ are obtuse, $\alpha(T)=2h'$ and $\alpha(S_\l(T))=2h$.
From elementary geometry it follows that $h'\leq h$ (see Figure \ref{fig:obtuseSteinerLeft}), so we have the desired inequality.\\
\textbf{Case 1b:} the triangle $S_\l(T)$ is acute.
Denote the length of the altitude of $T$ that lies on $\l$ by $h$, and the length of the inner altitude of $T$ (dropped from the obtuse angle) by $h'$.
Denote by $a$ the length of the side of $T$ the altitude $h$ is dropped to, and by $b$ the length of the longest edge of $T$, see Figure \ref{fig:obtuseSteiner}. 
Then $\alpha(T)=2h'$, from the formula for an area of a triangle $h'b=ha$, and since $T$ is obtuse, $b>\sqrt{a^2+h^2}$.
From here it follows that 
\[\alpha(T)=2h'=\frac{2ha}{b}<\frac{2ha}{\sqrt{a^2+h^2}}.\]
On the other hand, by the properties of symmetrization, $S_\l(T)$ is an isosceles triangle with base length $a$ and height $h$.
The fact that $S_\l(T)$ is acute means that $a<2h$.
By Lemma \ref{lem:alphaTriangle}
\[\alpha(S_\l(T))=2h\sin\theta,\]
where $\theta$ is the apex angle, and $\tan\frac{\theta}{2}=\frac{a}{2h}$.
Using the identity $\sin x=\frac{2\tan \frac{x}{2}}{\tan^2\frac{x}{2}+1}$, one can see that 
$\alpha(S_\l(T))=\frac{8ah^2}{a^2+4h^2}$.
Therefore, we only need to show that if $\frac{a}{2h}<1$ then
\[\frac{2ah}{\sqrt{a^2+h^2}} \leq \frac{8ah^2}{a^2+4h^2}.\]
If we write $t=\frac{a}{2h}$, then this inequality is equivalent to
\[\frac{4ht}{\sqrt{4t^2+1}}\leq\frac{4ht}{t^2+1},\]
which clearly holds for $0\leq t\leq 1$; this proves that $\alpha(T)\leq\alpha(S_\l(T))$.\\
\textbf{Case 2:} $\l$ is an inner altitude, in which case we have four sub-cases.\\
\textbf{Case 2a:} the triangles $T$ and $S_\l(T)$ are obtuse. 
In this case, $\l$ is an inner altitude both in $T$ and in $S_\l (T)$, and since both of these triangles are obtuse, then $\alpha(T)=\alpha(S_\l(T))$.\\
\textbf{Case 2b:} the triangle $T$ is obtuse and the triangle $S_\l (T)$ is acute.
We show that this case is actually impossible.
Consider the perpendicular bisector to the longest side of $T$.
This line is parallel to $\l$, so the Steiner symmetrization with respect to it, gives a translate of $S_\l(T)$, so we can replace $\l$ by the perpendicular bisector.
The triangle $T$ is obtuse.
This means that the vertex of the obtuse angle is inside a semi-disk the diameter of which is the longest edge.
Consider the chord in that disk which is parallel to the diameter and passes through the third vertex.
After symmetrization, the new third vertex will be in the midpoint of that chord, and hence again inside that semi-disk, making $S_\l (T)$ obtuse again, see Figure \ref{fig:obtuseAcuteImpossible}. \\
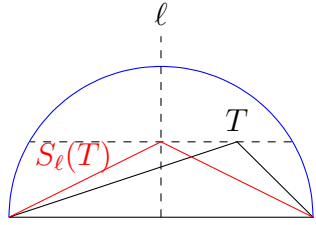
\begin{figure}
\centering
\begin{tikzpicture}[scale = 2]
\draw(-1,0)--(1,0)--(0.5,0.5)--(-1,0);
\draw[dashed] (0.86,0.5)--(-0.86,0.5);
\draw[domain = 0:180, smooth,blue, variable = \t] plot ({cos(\t)},{sin(\t)});
\draw[dashed](0,0)--(0,1.2);
\draw[red](1,0)--(0,0.5)--(-1,0);

\node[above] at (0,1.2) {$\l$};
\node[above] at (0.5,0.5) {$T$};
\node[left,red] at (-0.25,0.4) {$S_\l(T)$};
\end{tikzpicture}
\caption{Steiner symmetrization of an obtuse triangle in the direction of the inner altitude cannot be acute. \label{fig:obtuseAcuteImpossible}}
\end{figure}
\textbf{Case 2c:} the triangle $T$ is acute and the triangle $S_\l (T)$ is obtuse.
The line $\l$ is an inner altitude both in $T$ and in $S_\l(T)$, and the length of the altitude $h$ is the same in both triangles.
Therefore, by Lemma \ref{lem:alphaTriangle}, $\alpha(T)=2h\sin\theta$ and $\alpha(S_\l(T))=2h$, where $\theta$ is the angle at the vertex from which the altitude is dropped.
Consequently, $\alpha(T)\leq\alpha(S_\l(T))$.\\
\textbf{Case 2d:} the triangles $T$ and $S_\l(T)$ are acute.
The line $\l$ is an altitude both in $T$ and in $S_\l(T)$, and the length of the altitude, $h$, is the same in both triangles. 
Denote by $L$ the length of the edge to which this altitude is dropped.
This length is the same both in $T$ and in $S_\l(T)$.
Denote by $\theta$ the angle at the vertex from which the altitude is dropped in $T$, and by $\theta'$ the corresponding angle in $S_\l(T)$.
The triangle $S_\l(T)$ is isosceles, with $\theta'$ being the apex angle, so $\theta'=2\arctan\frac{L}{2h}$.
In $T$ the altitude divides the edge of length $L$ into segments of length $x$ and $L-x$, and thus $\theta(x)=\arctan\frac{x}{h}+\arctan\frac{L-x}{h}$, see Figure \ref{fig:acuteAcuteSymmetrization}.
It is elementary to check that for $0\leq x\leq L$, the function $\theta(x)$ reaches its maximum strictly at $x=\frac{L}{2}$, so $\theta\leq\theta'$.	
Then, from Lemma \ref{lem:alphaTriangle} it follows that $\alpha(T)=2h\sin\theta\leq 2h\sin\theta'=\alpha(S_\l(T))$. 
\begin{figure}
\centering
\begin{tikzpicture}[scale = 1.5]
\draw(0,0)--(1,0)--(0.8,2)--(0,0);
\draw(0.8,-0.1)--(0.8,2.1);
\draw[dashed,red](0.3,0)--(0.8,2)--(1.3,0)--(1,0);

\node[left] at (0,1) {$T$};
\node[right,red] at (1.3,1) {$S_\l(T)$};
\node[above] at (0.8,2.1) {$\l$};
\node[left] at (0.8,1) {$h$};
\node[below] at (0.2,0) {$L-x$};
\node[below] at (0.9,0) {$x$};
\end{tikzpicture}
\caption{The triangles $T$ and $S_\l(T)$ are both acute. \label{fig:acuteAcuteSymmetrization}}
\end{figure}
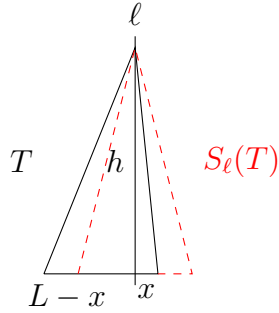
\end{proof}

Now, we present an alternative proof of Theorem \ref{thm:optimalTriangle} using Steiner symmetrization.
\begin{proof}[Proof of Theorem \ref{thm:optimalTriangle} using Steiner symmetrization]
Let $T$ be an arbitrary triangle in $\R^2$. 
As we explained, if $\l$ is an altitude of $T$ then $S_\l(T)$ is an isosceles triangle and $\beta(T)\leq \beta(S_\l(T))$.
Thus, we may assume that $T_0=T$ is isosceles, and is shifted so that the apex vertex is the origin.
Write $\l_0$ for one of the altitudes of $T_0$ which is not dropped from the apex, and let $T_1$ be the shift of $S_{\l_0}(T_0)$ for which the apex vertex is the origin. 
Repeat this construction to obtain a sequence of isosceles triangles $T_k$ such that for all $k$, $T_{k+1}$ is obtained by a shift from $S_{\l_{k}}(T_{k})$, where $\l_k$ is one of the altitudes of $T_k$ that is not dropped from the apex, and $T_{k+1}$ is shifted to have its apex vertex at the origin.
By Claim \ref{claim:steinerBettaTriangle}, for all $k$, $\beta(T_k)\leq \beta(T_{k+1})$.
We will show that the sequence $\set{T_k}$ has a subsequence that converges in the Hausdorff topology to some set $K$, and that this set $K$ must necessarily be a regular triangle. 
Since $\beta$ is continuous, then we would get $\beta(T)=\beta(T_0)\leq\beta(K)$, seeing how $\beta(K)$ is the limit of a subsequence of the monotone sequence $\beta(T_k)$. 
Thus we indeed get that the the billiard product of a triangle is bounded from above by that of a regular triangle.

Since we have $\diam(S_{\l_k}(T_k))\leq\diam(T_k)$, it follows that all the triangles $T_k$ are contained in the disk of radius $\diam(T_0)$ around the origin.
Thus we can use the Blaschke selection theorem, and conclude that there exists a subsequence of $\set{T_k}$ that converges to some convex body, $K$.
For convenience, we keep denoting this subsequence by $T_k$. 
Also, since Steiner symmetrization does not change the area, then all the triangles $T_k$ have the same area, call it $S$, and as a result, $K$ must also have the same area. 

\underline{Step 1:} We show that $K$ is a triangle.
Denote the vertices of $T_k$ by $a_k$, $b_k$, and $c_k$, where $a_k=0$ is the apex vertex. 
Then those sequences are contained in a closed disk, and thus, passing (iteratively) to a subsequence, they converge.
Again, for convience we may assume that the original sequences converge.
Write $a=0$, $b$, $c$ for the limits of those sequences.
Since the sequence $T_k$ converges to $K$, and $K$ is compact, then it follows that the points $a$, $b$, and $c$ are in $K$.
As a result, the convex hull of $a$, $b$, $c$, call it $K'$, is contained in $K$.
The area of a triangle can be written with an explicit formula in terms of its vertices, and hence the area of the triangle $K'$ is the limit of the areas of the triangles $T_k$, which is $S$.
The conclusion is the following: $K$ and $K'$ are convex bodies with the same area, and for which $K' \subseteq K$. 
This implies that $K=K'$, so $K$ is indeed a triangle, and its vertices are $a$, $b$, $c$.

\underline{Step 2:} We show that $K$ is a regular triangle.
For each $k$, we have the equality $|a_k-b_k| = |a_k-c_k|$.
Passing to the limit, we get $|a-b|=|a-c|$, so $K$ is an isosceles triangle, and its area is $S$.
If we show that the length of the base is $2\sqrt{S}\slash\sqrt[4]{3}$, then the height of the triangle must be $\sqrt[4]{3}\sqrt{S}$, and then elementary trigonometry would imply that the base angle is $\frac{\pi}{3}$, so $K$ is indeed a regular triangle.
Write $x_k$ for the length of the base of the triangle $T_k$ (here $k$ denotes the index of the original sequence, before taking any subsequences).
Note that from the properties of Steiner symmetrization, the leg length of $T_k$ is exactly $x_{k+1}$. 
This, together with the fact that the area of $T_k$ is $S$, and the Pythagorean theorem, shows that $x_k$ satisfies the following recursive relation:
\[x_{k+1}^2 = \frac{4S^2}{x_k^2}+\frac{x_k^2}{4}.\]
Write $y_k=x_k^2$ and $f(x)=\frac{4S^2}{x}+\frac{x}{4}$.
If $y_k$ converges, then the limit must be a positive fixed point of $f$, and this must be $L:=\frac{4S^2}{\sqrt{3}}$.
As a reuslt, we will get that $x_k$ converges to $\sqrt{L}=2\sqrt{S}\slash\sqrt[4]{3}$, as required.
It is simple to see that $f((0,\infty))=[2S,\infty)$, and that for $x \geq 2S$, we have $|f'(x)|\leq \frac{3}{4}<1$. 
This means that the ray $[2S,\infty)$ is contained inside the basin of attraction of the fixed point $L$, so for any initial condition $y_0$, we will have that $y_1=f(y_0)$ is inside the basin of attraction, an hence $y_k$ converges to $L$.

Lastly, we shall address the case of equality. 
Suppose that the original triangle $T$ is not regular.
If we examine carefully the proof of Claim \ref{claim:steinerBettaTriangle}, then we see that the inequality $\beta(T)\leq \beta(S_{\l}(T))$ proved there is strict, except maybe for two cases: both $T$ and $S_{\l}(T)$ are obtuse (case 2a), or if $T$ and $S_{\l}(T)$ are acute, and the altitude $\l$ is also a median (case 2d).
In the sequence we constructed, all triangles $T_k$ for $k\geq 2$ are acute.
If $T_2$ is not a regular triangle, then we will have a strict inequality,
\[\beta(T)\leq \beta(T_2)<\beta(T_3)\leq 12.\]
If $T_2$ is regular, and the original triangle $T$ is not, then necessarily one of the two symmetrization steps done so far must have increased the billiard product, so in any case the inequality is strict.
\end{proof}

\bibliography{bibliography}
\bibliographystyle{abbrv}
\end{document}